\documentclass[11pt]{amsart}

\usepackage{amsmath, amsbsy, amscd, amssymb, amsthm, ascmac, bm, enumerate}
\usepackage{graphicx}

\usepackage{etoolbox}
\patchcmd{\section}{\scshape}{\bfseries}{}{}
\makeatletter
\renewcommand{\@secnumfont}{\bfseries}
\makeatother

\theoremstyle{plain}
 \newtheorem{theorem}{Theorem}[section]
 
 \newtheorem{proposition}[theorem]{Proposition}
 
 \newtheorem{lemma}[theorem]{Lemma}
 
\theoremstyle{definition}
 \newtheorem{definition}[theorem]{Definition}
 \newtheorem{remark}[theorem]{Remark}
 
 \newtheorem{example}[theorem]{Example}
\renewcommand{\Re}{\operatorname{Re}}

\newcommand{\inner}[2]{\left\langle{#1},{#2}\right\rangle}

\numberwithin{equation}{section}
\numberwithin{figure}{section}

\title[T\MakeLowercase{riply periodic} ZMC \MakeLowercase{surfaces}]
      {\LARGE T\MakeLowercase{riply periodic zero mean curvature} \\[6pt]
      \MakeLowercase{surfaces in}
       L\MakeLowercase{orentz}-M\MakeLowercase{inkowski 3-space}}
\author[S.~F\MakeLowercase{ujimori}]
       {\large S\MakeLowercase{hoichi} F\MakeLowercase{ujimori}}
\address{Department of Mathematics \\
         Okayama University \\
         Tsushima-naka 3-1-1 \\
         Okayama 700-8530 \\ 
         Japan}
\email{fujimori@math.okayama-u.ac.jp}
\urladdr{http://www.math.okayama-u.ac.jp/~fujimori/}
\date{March 20, 2017}
\subjclass[2010]{Primary 53A10; Secondary 53A35, 53C50}
\keywords{zero mean curvature, triply periodic surface, fold singularity}
\thanks{
The author was partially supported by the Grant-in-Aid for Young
Scientists (B) No. 25800047 from Japan Society for the Promotion of Science.
}

\begin{document}

\begin{abstract}
We construct triply periodic zero mean curvature surfaces of mixed type 
in the Lorentz-Minkowski 3-space $\mathbb{L}^3$, with the same topology 
as the triply periodic minimal surfaces in the Euclidean 3-space $\mathbb{R}^3$, 
called Schwarz rPD surfaces. 
\end{abstract}

\maketitle

\section{Introduction}

Zero mean curvature surfaces in the Lorentz-Minkowski 3-space $\mathbb{L}^3$ 
are smooth surfaces with mean curvature zero wherever the mean curvature is 
defined.  
Having the mean curvature defined at all points is not expected, 
because these surfaces can change causal type, meaning that some parts may 
have spacelike tangent planes and other parts may have timelike tangent 
planes, with lightlike tangent planes at boundary points between these parts.  
An interesting aspect of these surfaces is precisely that they change 
causal type, often resulting in interesting singular and topological 
behaviors, and these surfaces have been investigated in 
\cite{Aka, FKKRSUYY2, FKKRUY, FKKRUY2, FRUYY3}.
One of the main tools for the construction of such surfaces 
is based on the fact that fold singularities of spacelike maximal surfaces 
have real analytic extensions to timelike minimal surfaces \cite{FKKRSUYY2}. 

In contrast to minimal surfaces in the Euclidean 3-space $\mathbb{R}^3$, 
the only known triply periodic zero mean curvature surfaces in $\mathbb{L}^3$ 
were those in a 1-parameter family constructed in \cite{FRUYY3}, 
while there are many known triply periodic minimal surfaces in $\mathbb{R}^3$, 
see for example \cite{FK, FW, Ka1, Sch}.

This motivates us to broaden our knowledge of triply periodic 
zero mean curvature surfaces in $\mathbb{L}^3$, 
and in this paper we construct a new 1-parameter family of triply periodic 
zero mean curvature surfaces in $\mathbb{L}^3$ based on the conjugate surfaces of 
the triply periodic minimal surfaces in $\mathbb{R}^3$ 
called the Schwarz H surfaces. 
This is the main original result of this paper, 
and the family here is interesting 
because it exhibits both a change of causal type, 
and also a greater complexity than the previously known examples.  
The method we use to construct the family is essentially the same as in 
\cite{FRUYY3}, but the surfaces here are less symmetric, and so the 
construction is more involved. 
It is expected that by using the method in this paper one could construct 
families of surfaces with more complicated topologies, based on the 
data of triply periodic minimal surfaces in $\mathbb{R}^3$ 
constructed in \cite{FW, Ka1}.
It is also expected that the family of surfaces we construct in this paper is 
a prototype for the study of the moduli space of triply periodic zero mean curvature surfaces. 

We remark that the surfaces constructed in this paper have the same topology 
and symmetry as Schwarz rPD minimal surfaces, 
not Schwarz H surfaces.  
(As for the symbols ``rPD'' and ``H'', see Remark~\ref{rm:tpms} 
in Appendix~\ref{sec:rpd}.)

\section{Preliminaries}
\label{sec:prelim}

We denote by $\mathbb{L}^3$ the Lorentz-Minkowsiki 3-space with indefinite metric 
$\inner{~}{~}=-dx_0^2+dx_1^2+dx_2^2$. 
Let $M$ be a Riemann surface. 
A conformal immersion $f:M\to\mathbb{L}^3$ is called 
a {\it spacelike surface} if the 
induced metric $ds^2=\inner{df}{df}$ is positive definite on $M$. 
A spacelike surface $f:M\to\mathbb{L}^3$ is called {\it maximal} if 
its mean curvature vanishes identically. 
In \cite{UY} a notion of {\it maxface} was introduced as a maximal surface with 
certain kind of singularities. 
More precisely, $f:M\to\mathbb{L}^3$ is called a {\it maxface} if there exists an 
open dense subset $W$ of $M$ such that the restriction $f|_W$ of $f$ to $W$ gives 
a conformal maximal immersion and $df(p)\ne 0$ for all $p\in M$. 

For maxfaces, a similar representation formula to Weierstrass representation for 
minimal surfaces in $\mathbb{R}^3$ is known. 

\begin{theorem}[Weierstrass-type representation \cite{K, UY}]\label{th:w-type-rep}
Let $(g,\,\eta )$ be a pair of a meromorphic function $g$ and a 
holomorphic differential $\eta$ on a Riemann surface $M$ so that 
$(1+|g|^2)^2\eta \bar{\eta}$
gives a Riemannian metric on $M$. 
We set 
\begin{equation}\label{eq:maxPhi}
\Phi = \begin{pmatrix}
          -2g\eta \\ (1+g^2)\eta \\ i(1-g^2)\eta\
        \end{pmatrix},
\end{equation}
where $i=\sqrt{-1}$.  
Then
\begin{equation}\label{eq:maxface}
f=\Re\int_{z_0}^z \Phi:M\to\mathbb{L}^3\qquad (z_0\in M)
\end{equation}
defines a maxface. 
The singular set $S(f)$ of $f$ is given by 
\[
S(f)=\{p\in M\;;\;|g(p)|=1\}.
\] 
Moreover, $f$ is single-valued on $M$ if and only if 
\begin{equation}\label{eq:period}
\mathrm{Re} \oint_\ell
\Phi =\bm{0}
\end{equation}
for any closed curve $\ell$ on $M$. 
Conversely, any maxface can be obtained in this manner. 
\end{theorem}

The pair $(g,\,\eta)$ in Theorem~\ref{th:w-type-rep} is called 
the {\it Weierstrass data} of $f$. 

\begin{remark}
The first fundamental form $ds^2$ and the second fundamental form 
${\rm I}\!{\rm I}$ of the surface \eqref{eq:maxface} are given by
\[
ds^2=\left( 1-|g|^2\right)^2\eta\bar\eta , \qquad
{\rm I}\!{\rm I}=-\eta dg-\overline{\eta dg}.
\]
Moreover, 
$g|_{M\setminus S(f)}:M\setminus S(f)
\to(\mathbb{C}\cup\{\infty\})\setminus \{|z|=1\}$ 
coincides with the composition of the Gauss map 
\[
G|_{M\setminus S(f)}:M\setminus S(f)
\to H^2=\{x\in\mathbb{L}^3\;;\;\inner{x}{x}=-1\}
\]
of the maximal surface and the stereographic projection 
\[
\sigma:H^2\ni (x_0,x_1,x_2)
\mapsto\frac{x_1+ix_2}{1-x_0}\in\mathbb{C}\cup\{\infty\},
\]
that is, 
$g|_{M\setminus S(f)}=\sigma\circ G|_{M\setminus S(f)}$.
So we call $g$ the Gauss map of the maxface. 
\end{remark}

Generic singularities of maxfaces are classified in \cite{FSUY}. 
Moreover several criteria for singular points of maxfaces by using their 
Weierstrass data are given in \cite{FRUYY2, FSUY, UY}. 

\begin{definition}[Fold singular point \cite{FKKRSUYY2}]\label{df:fold}
Let $f:M\to\mathbb{L}^3$ be a maxface with Weierstrass data $(g,\eta)$. 
We denote by $S(f)$ the singular set of $f$, that is, 
$S(f)=\{p\in M\;;\;|g(p)|=1\}$.
\begin{enumerate}
\item A singular point $p\in S(f)$ of $f$ is called {\it non-degenerate} if 
$dg$ does not vanish at $p$. 
\item A regular curve $\hat\gamma$ on $M$ is called a 
{\it non-degenerate fold singularity} if it consists of non-degenerate singular 
points such that
\[
\Re\left(\frac{dg}{g^2\eta}\right)
\] 
vanish identically along the curve $\hat\gamma$. 
Each point on the non-degenerate fold singularity is called a 
{\it fold singular point}. 
\end{enumerate}
\end{definition}

\begin{remark}
Let $\Sigma$ be a smooth 2-manifold and $f:\Sigma\to\mathbb{L}^3$ a smooth map. 
A singular point $p\in\Sigma$ of $f$ is called {\it fold singularity} if there 
exist a local coordinate system $(U;\varphi)$ ccentered at $p\in\Sigma$ and a 
diffeomorphism $\psi$ of $\mathbb{L}^3$ such that 
$(\psi\circ f\circ\varphi^{-1})(u,v)=(u,v^2,0)$. 

In~\cite[Lemma~2.17]{FKKRSUYY2}, 
it is shown that a non-degenerate fold singularity of a maxface is 
indeed a fold singularity.
\end{remark}

A regular curve $\gamma:(a,b)\to \mathbb{L}^3$
is called {\it null\/} or {\it isotropic\/}
if $\gamma'(t)=(d\gamma/dt)(t)$ is a lightlike vector 
for all $t\in (a,b)$.
\begin{definition}[Non-degenerate null curve \cite{FKKRSUYY2}]
 A null curve $\gamma:(a,b)\to \mathbb{L}^3$
 is called {\it degenerate} or {\it non-degenerate\/} at $t=c$
 if $\gamma''(c)$ is or is not proportional
 to the velocity vector $\gamma'(c)$, respectively.
 If $\gamma$ is non-degenerate at each $t\in (a,b)$,
 it is called a {\it non-degenerate\/} null curve.
\end{definition}

\begin{theorem}[Analytic extension of maxface \cite{FKKRSUYY2}]
 Let $f:M\to \mathbb{L}^3$ be a maxface
 which has non-degenerate fold singularities 
 along a singular curve $\hat\gamma:(a,b)\to M$.
 Then $\gamma=f\circ \hat\gamma$ is a 
 non-degenerate null curve and the image of the map
 \begin{equation}\label{eq:Bj}
   f^*(u,v)=\frac{\gamma(u+v)+\gamma(u-v)}{2}
 \end{equation}
 is real analytically connected to the image of
 $f$ along $\gamma$ as a timelike minimal immersion.
 Conversely, any real analytic immersion with mean curvature, 
 whereever well-defined, equal to zero
 which changes type across a non-degenerate
 null curve is obtained as a real analytic extension
 of non-degenerate fold singularities of a maxface.
\end{theorem}

We call an immersion in $\mathbb{L}^3$ with mean curvature, 
whereever well-defined, equal to zero 
a {\it zero mean curvature $($ZMC$\,)$ surface}.

\section{Schwarz H-type ZMC surfaces}
\label{sec:schwarz-h-zmc}

For a constant $a\in (0,\infty)$, we set $M_a$ a Riemann surface 
of genus 3 defined by the hyperelliptic curve
\[
  w^2=z(z^3+a^3)(z^3+a^{-3}). 
\]
Consider the family $\{f_a\}_{0<a<\infty}$ of maxfaces 
\begin{equation}\label{eq:w-rep-max}
 f_a=\begin{pmatrix}x_0 \\ x_1 \\ x_2\end{pmatrix}
    =\Re\int
   \begin{pmatrix}
     -2g \\ 1+g^2 \\ i(1-g^2) 
   \end{pmatrix} \eta
\end{equation}
in $\mathbb{L}^3$ with the Weierstrass data
\begin{equation}\label{eq:w-data-conj}
  g=z,\qquad \eta=i\frac{dz}{w}.
\end{equation}
The singular set of $f_a$ is $\{|z|=1\}$. 
It is easy to verify that each singular point is fold singularity.

We define a $\mathbb{C}^3$-valued 1 form $\Phi_a$ on $M_a$ by
\[
  \Phi_a =\begin{pmatrix}
           -2g \\ 1+g^2 \\ i(1-g^2) 
           \end{pmatrix} \eta .
\]
Direct computations show the following lemma.

\begin{lemma}[Symmetries of the surface]
Define anti-holomorphic maps $\psi_j :M_a\to M_a$ $(j=1,2,3)$ as follows:
\begin{align*}
\psi_1(z,w)&=(\bar{z},\bar{w}), \\
\psi_2(z,w)&=(e^{2\pi i/3}\bar{z},e^{\pi i/3}\bar{w}), \\
\psi_3(z,w)&=\left(\frac{1}{\bar{z}},\frac{\bar{w}}{\bar{z}}\right).
\end{align*}
Then we have the following:
\begin{align*}
\psi_1^*\Phi_a&=
   \begin{pmatrix}
     -1 & 0 & 0 \\
      0 &-1 & 0 \\
      0 & 0 & 1
   \end{pmatrix}\overline{\Phi}_a, \\
\psi_2^*\Phi_a&=
   \begin{pmatrix}
      1 & 0           & 0 \\
      0 &-\cos(\pi/3) & \sin(\pi/3) \\
      0 & \sin(\pi/3) & \cos(\pi/3)
   \end{pmatrix}\overline{\Phi}_a, \\
\psi_3^*\Phi_a&=\overline{\Phi}_a.
\end{align*}
\end{lemma}

By the above lemma, we can consider 
\[
  \Omega_a^{\mathrm max}=\{f_a(z)\;;\;|z|\le 1,\;0\le \arg z \le \pi/3\}
\]
as the fundamental piece of the maxface, that is, the entire maxface consists of 
pieces each of which is congruent to $\Omega_a^{\mathrm max}$. 

\begin{lemma}
In $\Omega_a^{\mathrm max}$, the images of 
$\{0\le |z|\le 1,\;\arg z=0\}$ and 
$\{a\le |z|\le 1,\;\arg z=\pi/3\}$ by $f_a$ are straight lines, 
and the image of  
$\{a\le |z|\le 1,\;\arg z=\pi/3\}$ by $f_a$ is a curve in some timelike plane. 
\end{lemma}

\begin{proof}
Consider the Hopf differential 
\[
Q=\eta dg = i\frac{dz^2}{w}
\]
of $f_a$. 
If $z=t$ ($0\le t \le 1$), then $Q\in i\mathbb{R}$. 
Also, if $z=e^{\pi i/3}t$ ($0\le t \le 1$), then
\[
Q=\frac{-dt^2}{\sqrt{t(t^3-a^3)(t^3-a^{-3})}} \in
\begin{cases}
 \mathbb{R} & (0\le t \le a) \\
i\mathbb{R} & (a\le t \le 1)
\end{cases}
\]
this completes the proof. 
\end{proof}

Next we consider the singular curve $\gamma$ of $f_a$. 
The singular curve is the image of $z=e^{it}$ ($0\le t \le \pi/3$).   
Hence we can write 
\[
\gamma (s) =\!\int_0^s\!
        \begin{pmatrix}
         1 \\ -\cos t \\ -\sin t
        \end{pmatrix}\xi(t) dt, 
\;\;
\xi(t)=\frac{2}{\sqrt{2\cos 3t +a^3+a^{-3}}}
\;\; \left( 0\le s\le\frac{\pi}{3}\right)
\]
by a direct computation. 
Thus if we set 
\begin{equation}\label{eq:timef}
f_a^*(u,v)=\frac{1}{2}\big(\gamma(u+v)+\gamma(u-v)\big),
\end{equation}
then $f_a^*$ is a timelike minimal surface such that 
$\{v=0\}$ corresponds to the fold singularities 
and $f_a^*$ is the analytic extension of the maximal surface $f_a$. 

Arguments similar to those in \cite{FRUYY3} show the following two lemmas. 

\begin{lemma}[{\cite[Lemma 3.1]{FRUYY3}}]
$f_a^*(u,v)$ is an immersion on $(u,v)\in\mathbb{R}\times (0,\pi)$. 
\end{lemma}

\begin{lemma}[{\cite[Lemma 3.2]{FRUYY3}}]
$f_a^*(0,v)$ $(0<v<\pi)$ is a straight line parallel to $x_2$-axis, and 
$f_a^*(\pi/3,v)$ $(0<v<\pi)$ is a straight line parallel to 
$x_0=x_1+\sqrt{3}x_2=0$. 
\end{lemma}

Moreover, since $f_a^*(u,\pi+v)=f_a^*(u,\pi-v)$ holds, 
we have the following lemma. 

\begin{lemma}
$f_a^*(u,\pi)$ $(u\in\mathbb{R})$ corresponds to fold singularities. 
\end{lemma}

We set 
\(
  \Omega_a^{\mathrm min}=\{f_a^*(u,v)\;;\;0\le u \le \pi/3,\;0\le v \le \pi\}.
\)

\begin{remark}
For the Schwarz D-type ZMC surface in \cite{FRUYY3}, 
$f_a^*(u,\pi/2)$ is a straight line parallel to $x_0$-axis, 
but we do not have such a symmetry in this case. 
\end{remark}

We set 
\[
\sigma (s) =f_a^*(s,\pi) 
            =\frac{1}{2}\big(\gamma(s+\pi) + \gamma(s-\pi)\big)
\qquad (0\le s \le \pi/3)
\]
to further extend analytically from $f_a^*(u,\pi)$ to spacelike surface.  
Then we have 
\[
\sigma'(s) 
=\begin{pmatrix}
  1 \\ \cos s \\ \sin s
 \end{pmatrix}\hat\xi(s), 
\]
where
\[
\hat\xi(s)=\xi(s+\pi)=\xi(s-\pi)
=\frac{2}{a^3+a^{-3}-2\cos 3s}. 
\]
A direct computation shows the following lemma. 

\begin{lemma}
The following equation  
\[
\sigma'(s)=A\gamma'\left(\frac{\pi}{3}-s\right) 
\]
holds, where 
\[
A=\begin{pmatrix}
      1 & 0           & 0 \\
      0 &-\cos(\pi/3) &-\sin(\pi/3) \\
      0 &-\sin(\pi/3) & \cos(\pi/3)
   \end{pmatrix}.
\]
\end{lemma}

By this lemma, we have  
\[
\sigma(s)=A\gamma\left(\frac{\pi}{3}-s\right) +\bm{c},
\]
where
\[
\bm{c}=\sigma(0)-A\gamma(\pi/3)
      =f_a^*(0,\pi)-Af_a^*(\pi/3,0)\in\mathbb{L}^3.
\]
Thus we have the following proposition (See Fig.~\ref{fg:h-fp}). 

\begin{proposition}
We denote by $\hat{f}_a$ the spacelike extension from $\sigma(s)$.
Then we have 
\[
\hat{f}_a(z)=-Af_a(z) +\bm{c}
\qquad (|z|\le 1, \;\; 0\le \arg z\le \pi/3).
\]
\end{proposition}

%
We set 
\(
  \hat\Omega_a^{\mathrm max}
  =\{\hat{f}_a(z)\;;\;|z|\le 1,\;0\le \arg z \le \pi/3\}. 
\)
Then the boundary of 
\begin{equation}\label{eq:fund-p}
\Omega_a^{\mathrm max}\cup\Omega_a^{\mathrm min}\cup\hat\Omega_a^{\mathrm max}
\end{equation}
consists of two planar curves and two straight lines. 
See Fig.~\ref{fg:h-fp}. 
Now we extend this piece \eqref{eq:fund-p} by reflections with respect to 
planar curves, 
then six copies of \eqref{eq:fund-p} look like ``twisted'' equilateral triangular 
catenoid, see Fig.~\ref{fg:schwarzh-mix}. 
This triangular catenoid is homeomorphic and has the same symmetry to 
the half of rPD family (in $\mathbb{R}^3$) as in Example~\ref{sbsec:rpd}. 
Therefore the ZMC surface we obtain by extending \eqref{eq:fund-p} 
by reflections infinitely many times is triply periodic. 
Though the triply periodic ZMC surface looks like embedded for any $a\in (0,1)$, 
we leave the study of embeddedness of this family for another occasion. 
See Fig.~\ref{fg:schwarzh-mix}. 

\begin{figure}[htbp] 
\begin{center}
\begin{tabular}{cc}
 \includegraphics[width=.35\linewidth]{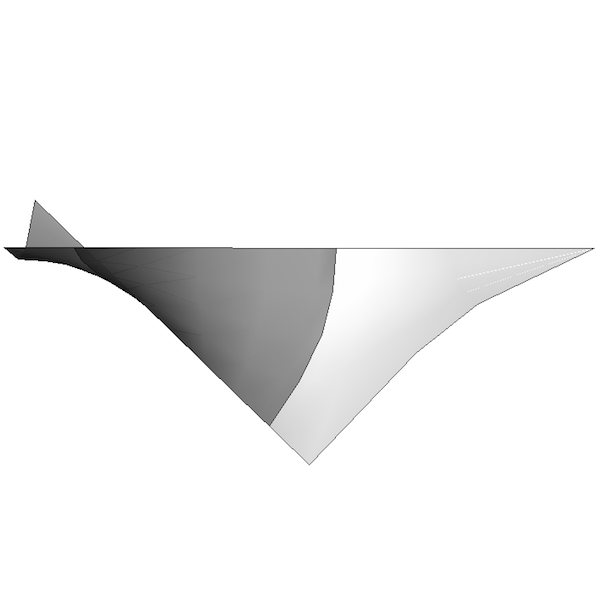} &
 \includegraphics[width=.35\linewidth]{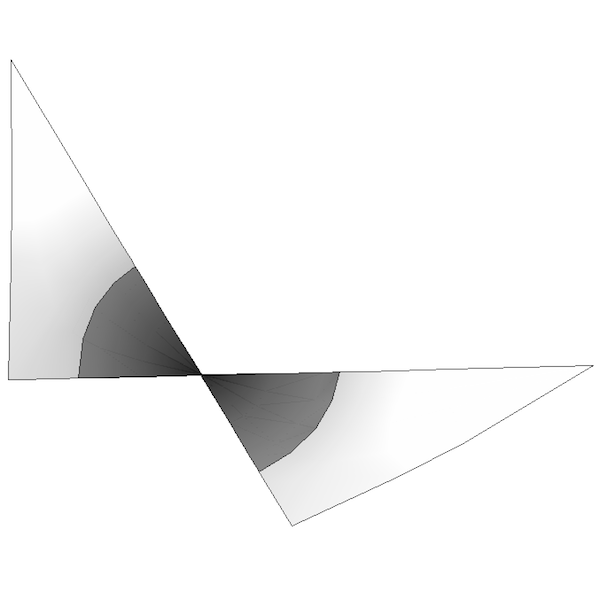} \\
 \multicolumn{2}{c}{$a=0.5$} 
\end{tabular}
\caption{The piece 
$\Omega_a^{\mathrm max}\cup\Omega_a^{\mathrm min}\cup\hat\Omega_a^{\mathrm max}$ 
in different view points. 
The spacelike parts are indicated by grey shades, 
and the timelike part is indicated by black shade. }
\label{fg:h-fp}
\end{center}
\end{figure} 

\begin{figure}[htbp] 
\begin{center}
\begin{tabular}{cc}
 \includegraphics[width=.35\linewidth]{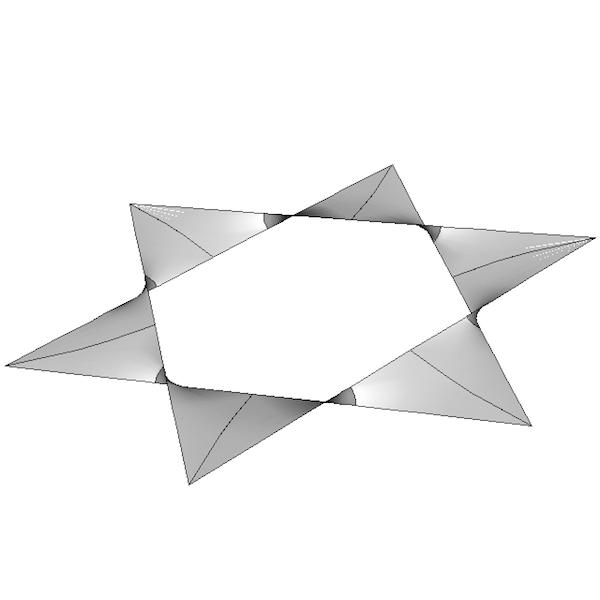} &
 \includegraphics[width=.35\linewidth]{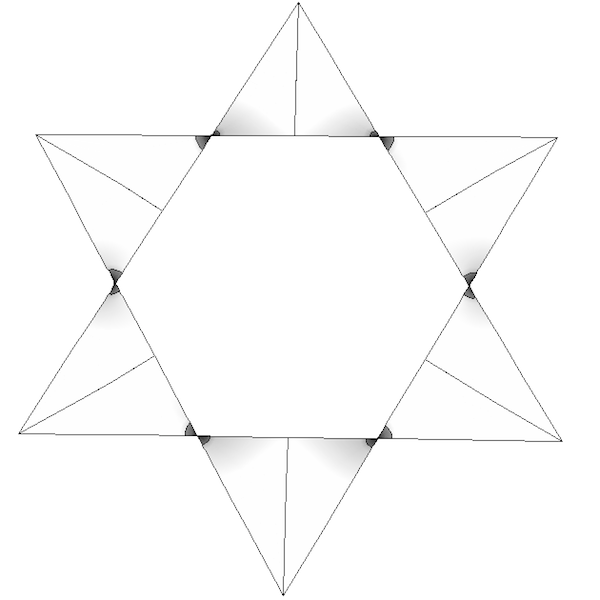} \\
 \multicolumn{2}{c}{$a=0.1$} \\
 \includegraphics[width=.35\linewidth]{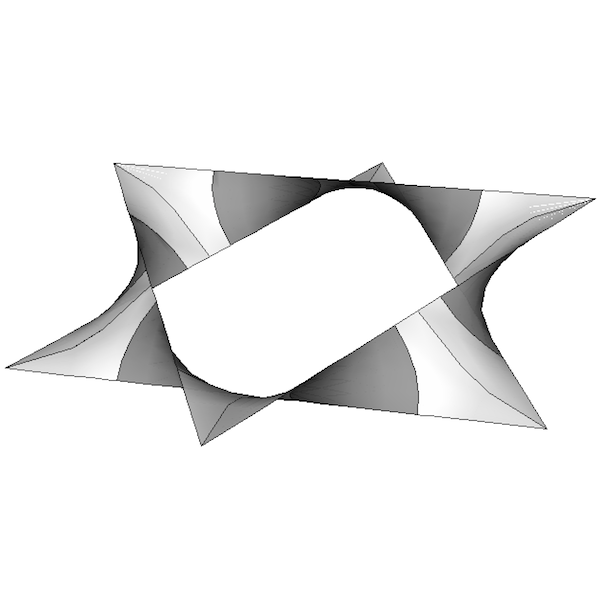} &
 \includegraphics[width=.35\linewidth]{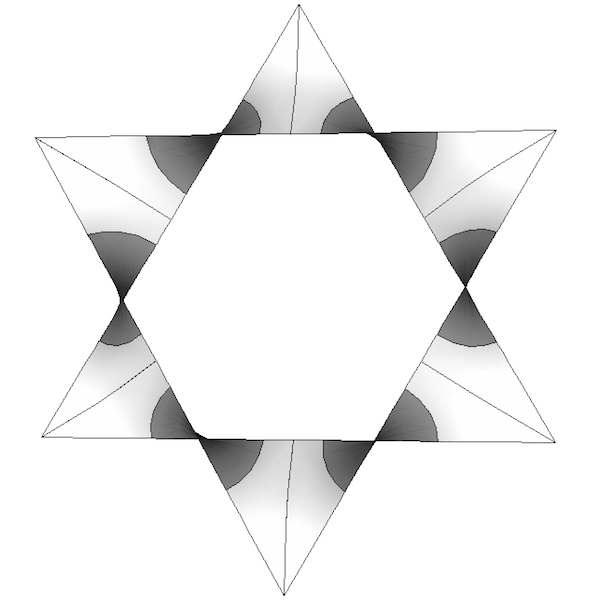} \\
 \multicolumn{2}{c}{$a=0.5$} \\
 \includegraphics[width=.35\linewidth]{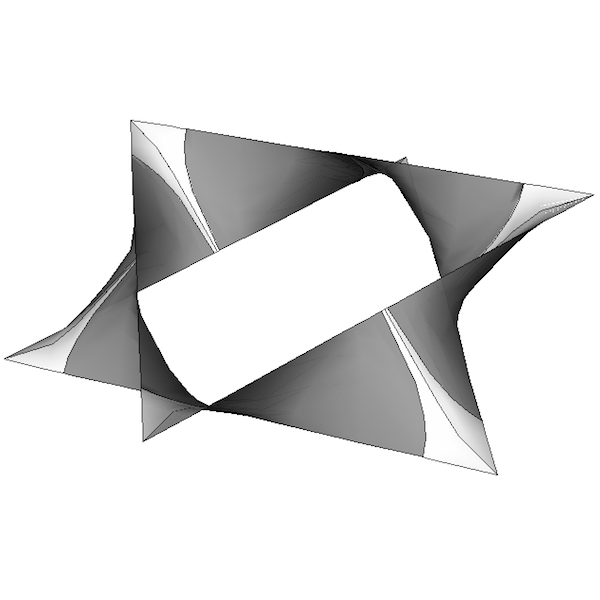} &
 \includegraphics[width=.35\linewidth]{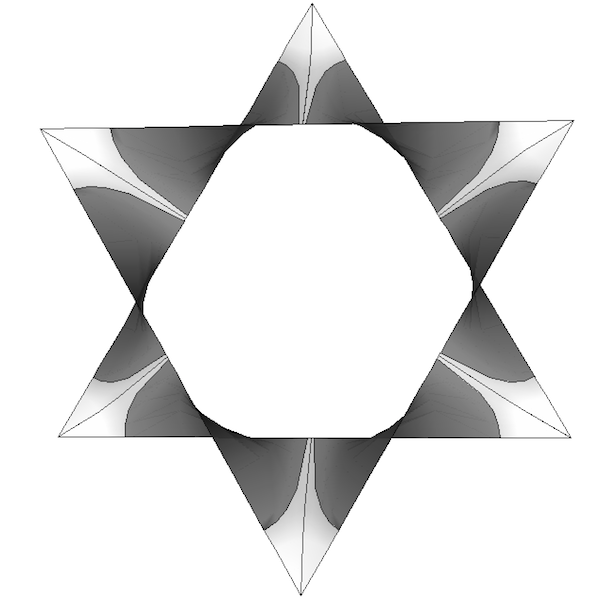} \\
 \multicolumn{2}{c}{$a=0.9$}
\end{tabular}
\caption{Schwarz H-type ZMC surfaces. }
\label{fg:schwarzh-mix}
\end{center}
\end{figure} 

We call the 1 parameter family of this triply periodic ZMC surface by 
{\em Schwarz H-type ZMC surfaces}. 

\begin{remark}
The 1 parameter family of the conjugate surface of the maxface  
we have considered in this section, that is, the maxface with the Weierstrass data 
$g=z$, $\eta=dz/w$, have conelike singularities, 
and the half of the fundamental piece looks like “twisted ”
equilateral triangular Lorentzian catenoid.
See Fig.~\ref{fg:schwarzh-conj}. 
Hence by extending these surfaces by reflections with respect to boundary 
straight lines, 
we have triply periodic maxfaces with conelike singularities. 
\end{remark}

\begin{figure}[htbp] 
\begin{center}
\begin{tabular}{cc}
 \includegraphics[width=.35\linewidth]{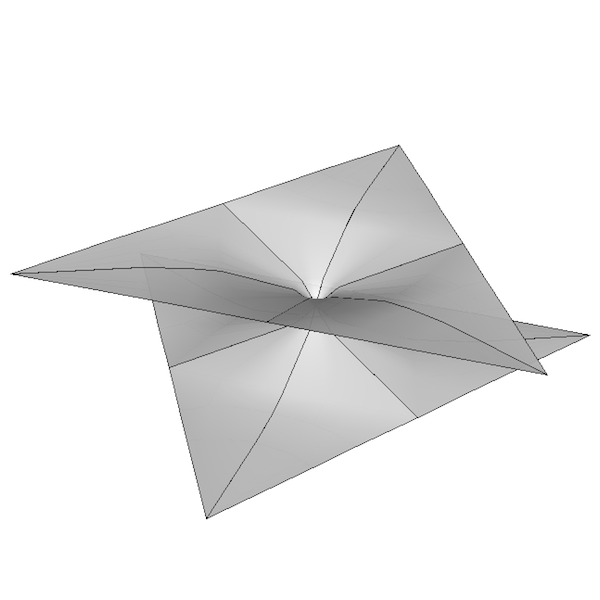} &
 \includegraphics[width=.35\linewidth]{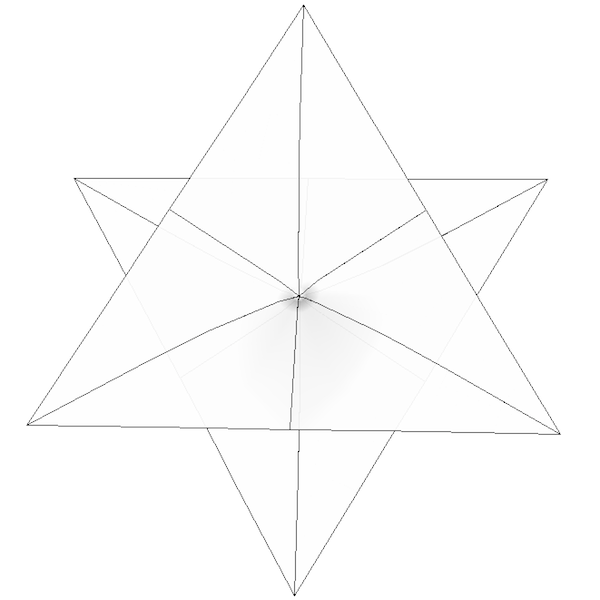} \\
 \multicolumn{2}{c}{$a=0.1$} \\
 \includegraphics[width=.35\linewidth]{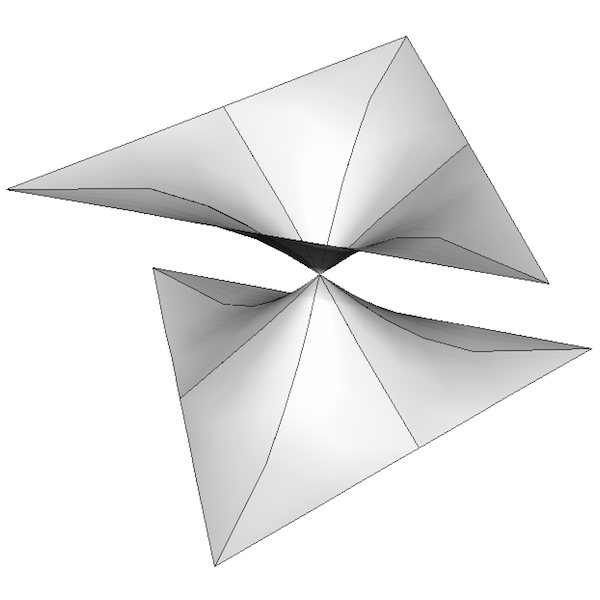} &
 \includegraphics[width=.35\linewidth]{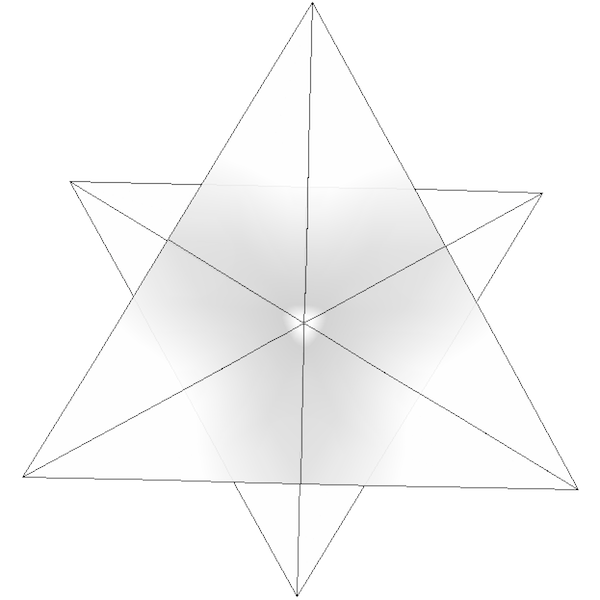} \\
 \multicolumn{2}{c}{$a=0.5$} \\
 \includegraphics[width=.35\linewidth]{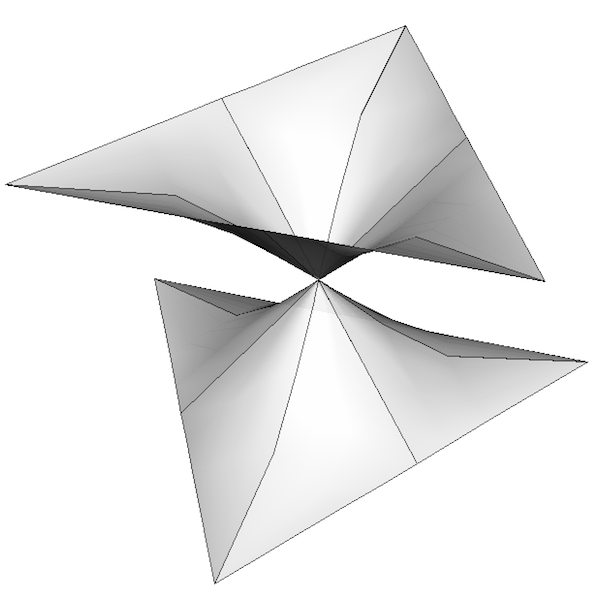} &
 \includegraphics[width=.35\linewidth]{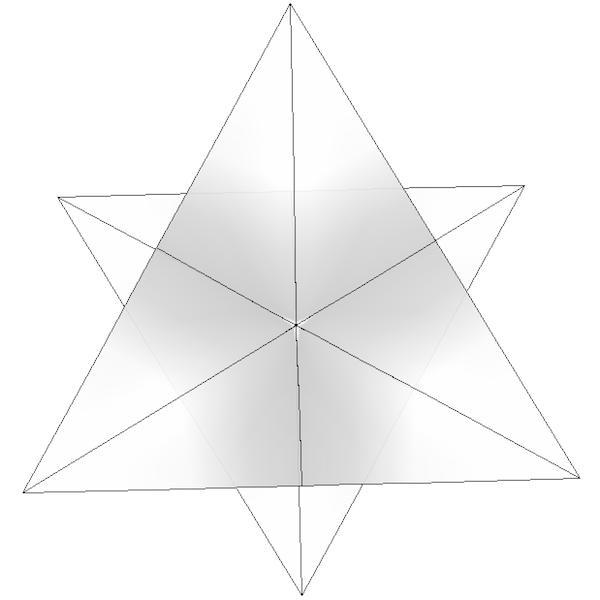} \\
 \multicolumn{2}{c}{$a=0.9$}
\end{tabular}
\caption{The conjugate surfaces of Schwarz H-type surface.}
\label{fg:schwarzh-conj}
\end{center}
\end{figure} 

\section{Limits of Schwarz H-type ZMC surfaces}
\label{sec:lim-h}

In this section we consider the limits of Schwarz H-type ZMC surfaces.
As $a\to 0$, the surface, with rescaled by $\sqrt{a^3+a^{-3}}$, 
converges to the helicoid by the same arguments as in \cite[Remark 3.6]{FRUYY3}. 

Next we consider the limit as $a\to 1$. 
Since the hyperelliptic curve $w^2=z(z^3+a^3)(z^3+a^{-3})$ converges to 
\[
w^2 = z(z^3+1)^2,
\]
the Riemann surface $M_a$ converges to a Riemann surface with six nodes at 
\[
 z=e^{\pi i/3}, \;\;-1,\;\;  e^{-\pi i/3}
\]
and two branch points at 
\[
z=0,\;\;\infty .
\]
This Riemann surface is of genus zero with six nodal singular points. 
Hence the maxface $f_a$ converges to 
\begin{equation}\label{eq:a-1-lim}
f_a\to \pm \Re\int
      \begin{pmatrix}
        -2z \\ 1+z^2 \\i(1-z^2)
      \end{pmatrix}
      \dfrac{i\,dz}{\sqrt{z}(z^3+1)}.
\end{equation}
Let $\zeta$ be a branch of $\zeta^2=z$. 
Then $\zeta$ is a coordinates of this Riemann surface, with six nodes at 
$\zeta =\pm e^{\pm\pi i/6},\;\pm i$, and the right hand side of 
\eqref{eq:a-1-lim} becomes 
\[
\pm 2 \Re\int
      \begin{pmatrix}
        -2\zeta^2 \\ 1+\zeta^4 \\i(1-\zeta^4)
      \end{pmatrix}
      \dfrac{i\,d\zeta}{\zeta^6+1}. 
\]
This surface coincides with the Karcher-type maxface with $k=3$, 
which is a maxface obtained by the conjugate of the maxface with 
the Weierstrass data as in Example~\ref{sbsec:karcher}. 
See Fig.~\ref{fg:schwarzh-karcher}.  

\begin{figure}[htbp] 
\begin{center}
\begin{tabular}{cc}
 \includegraphics[width=.35\linewidth]{figures/schwarzh-mix09g.png} &
 \includegraphics[width=.35\linewidth]{figures/schwarzh-mix09t.png} \\
 \multicolumn{2}{c}{$a=0.9$} \\
 \includegraphics[width=.35\linewidth]{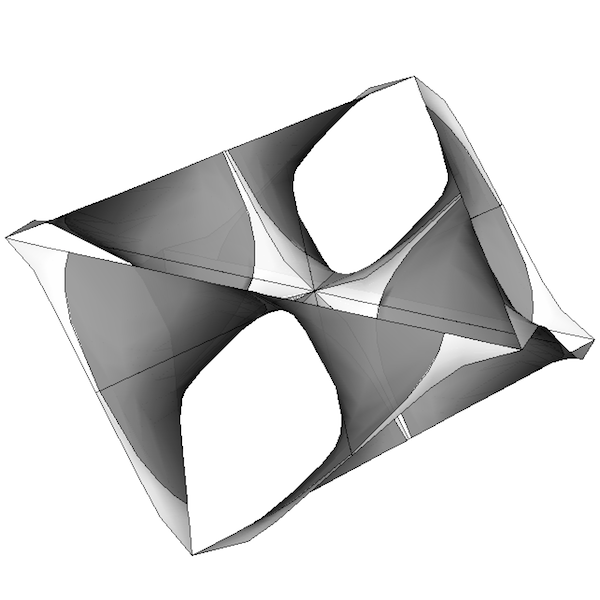} &
 \includegraphics[width=.35\linewidth]{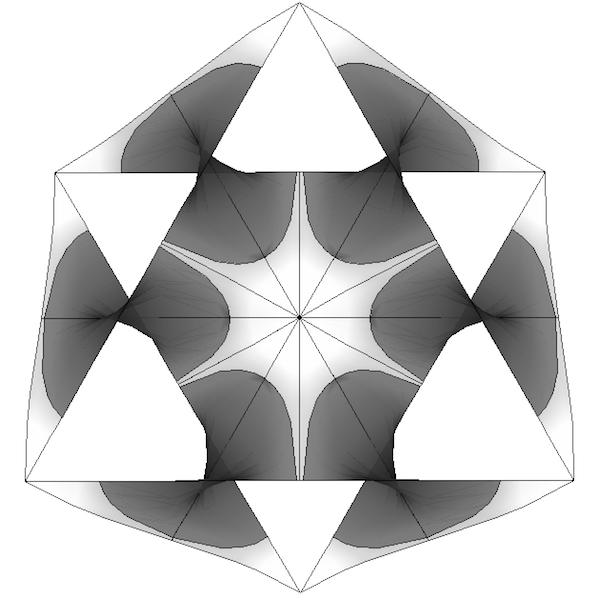} \\
  \multicolumn{2}{c}{$a=0.9$ (the same surface as above with different lattice)} \\
 \includegraphics[width=.35\linewidth]{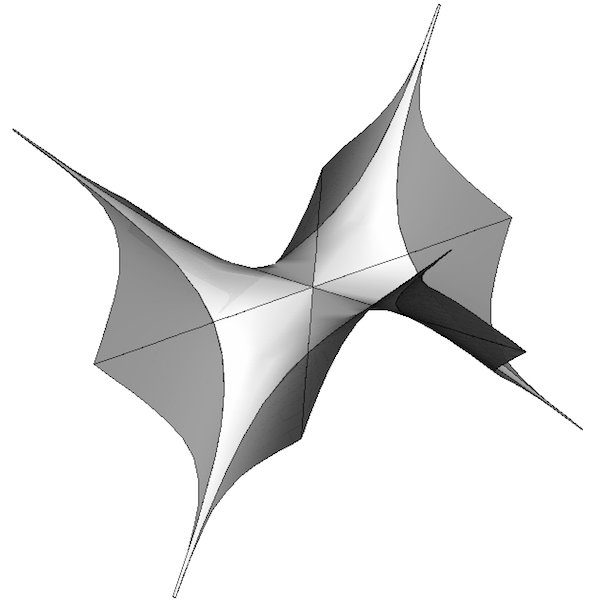} &
 \includegraphics[width=.35\linewidth]{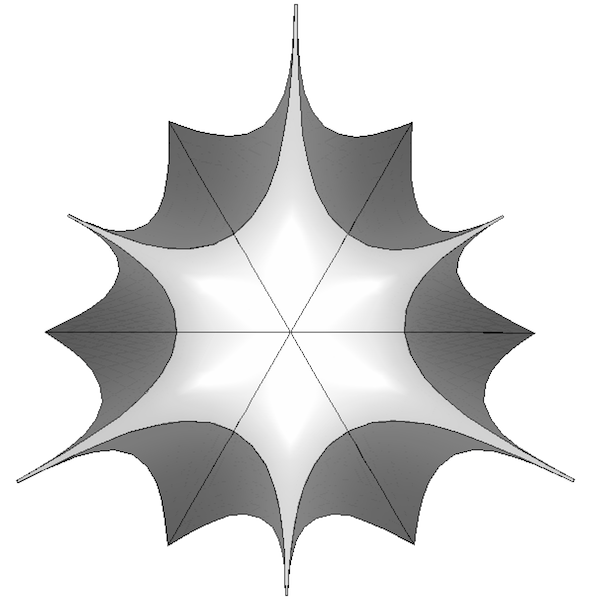} \\
 \multicolumn{2}{c}{$a\to 1$}
\end{tabular}
\caption{The limit of Schwarz H-type ZMC surface as $a\to 1$.}
\label{fg:schwarzh-karcher}
\end{center}
\end{figure} 

\begin{remark}
In contrast to the Karcher tower in $\mathbb{R}^3$, 
the Karcher-type maxface is single-valued on $M_k$ for any $k\ge 2$.  
Moreover, it is easy to verify that each singular point is fold singularity.  

For $k=2$, the image of the analytic extension of the maxface to ZMC surface 
coincides with the entire graph 
\begin{equation}\label{eq:scherk-graph}
x_0=\log\frac{\cosh x_1}{\cosh x_2}
\end{equation}
which is called the Scherk-type ZMC surface. 
\end{remark}

\appendix

\section{Minimal surfaces in $\mathbb{R}^3$}
\label{sec:rpd}

Here we review several examples of minimal surfaces in $\mathbb{R}^3$ 
which are related to ZMC surfaces we constructed in this paper. 
For the detail of these examples, see for example \cite{EFS, FW, Ka1, Ka2}.

\begin{theorem}[Weierstrass representation \cite{O}]\label{th:w-rep}
Let $(g,\,\eta )$ be a pair of a meromorphic function $g$ and a 
holomorphic differential $\eta$ on a Riemann surface $M$ so that 
\begin{equation}\label{eq:1stff}
(1+|g|^2)^2\eta \bar{\eta}
\end{equation}
gives a Riemannian metric on $M$. 
We set 
\begin{equation}\label{eq:Phi}
\Phi = \begin{pmatrix}
          (1-g^2)\eta \\ i(1+g^2)\eta\\ 2g\eta\
        \end{pmatrix}.
\end{equation}
Then
\begin{equation}\label{eq:surf}
f=\Re\int_{z_0}^z \Phi:M\to\mathbb{R}^3\qquad (z_0\in M)
\end{equation}
defines a conformal minimal immersion. 
Moreover, $f$ is single-valued on $M$ if and only if 
\begin{equation}\label{eq:period-min}
\mathrm{Re} \oint_\ell\Phi =\bm{0}
\end{equation}
for any closed curve $\ell$ on $M$. 
Conversely, any minimal surface can be obtained in this manner. 
\end{theorem}

The pair $(g,\,\eta)$ in Theorem~\ref{th:w-rep} is called the {\it Weierstrass data} of $f$. 
\begin{remark}
To verify the periodicity of surfaces, consider the following map. 
\begin{equation}\label{eq:perf}
\mathrm{Per}(f)=
\left\{\mathrm{Re}\oint_{\ell}
        \Phi\;;\;
       \ell\in H_1(M,\Bbb{Z})\right\}.
\end{equation}
The periodicity can be determined in the following way: 
\begin{itemize}
\item If $\mathrm{Per}(f)=\{\bm{0}\}$,  that is, 
      $f$ satisfies the condition (\ref{eq:period-min}) for any closed curve 
      $\ell$ on $M$, 
      then $f:M\to\Bbb{R}^3$ is well-defined on $M$, 
      that is, $f$ is \textit{non-periodic}. 
\item If there exists only one direction 
      $\bm{v}\in\Bbb{R}^3\setminus\{\bm{0}\}$ such that  
      \[
      \mathrm{Per}(f)\subset\Lambda_1=\{n\bm{v}\;;\;n\in\Bbb{Z}\},
      \] 
      then $f$ is \textit{singly periodic}.  In this case, 
      $f$ is well-defined in $\Bbb{R}^3/\Lambda_1\approx\Bbb{R}^2\times S^1$. 
      (A surface invariant under \textit{screw-motions} $\Lambda_1+R$, 
      where $R$ is a rotation around an axis in the direction of $\Lambda_1$, 
      is also singly periodic.  
      See, for example, \cite{CHM} and the references therein.)
\item If there exist two linearly independent vectors 
      $\bm{v}_1,\bm{v}_2\in\Bbb{R}^3$ (with $\mathrm{span}\{\bm{v}_1,\bm{v}_2\}$ 
      uniquely determined) such that
      \[ 
      \mathrm{Per}(f)\subset
      \Lambda_2=
      \left\{\sum_{j=1}^2n_j\bm{v}_j\;;\;n_j\in\Bbb{Z}\right\},
      \]
      then $f$ is \textit{doubly periodic}.  In this case, 
      $f$ is well-defined in $\Bbb{R}^3/\Lambda_2\approx T^2\times\Bbb{R}$.  
\item If there exist three linearly independent vectors 
      $\bm{v}_1,\bm{v}_2,\bm{v}_3\in\Bbb{R}^3$ such that  
      \[
      \mathrm{Per}(f)\subset
      \Lambda_3=
      \left\{\sum_{j=1}^3n_j\bm{v}_j\;;\;n_j\in\Bbb{Z}\right\},
      \]
      then $f$ is \textit{triply periodic}.  In this case, 
      $f$ is well-defined in $\Bbb{R}^3/\Lambda_3\approx T^3$.  
\end{itemize}
\end{remark}

\begin{remark}
The first fundamental form $ds^2$ and the second fundamental form 
${\rm I}\!{\rm I}$ of the surface \eqref{eq:surf} are given by
\[
ds^2=\left( 1+|g|^2\right)^2\eta\bar\eta , \qquad
{\rm I}\!{\rm I}=-\eta dg-\overline{\eta dg}.
\]
Moreover, $g:M\to\mathbb{C}\cup\{\infty\}$ coincides with the composition of 
the Gauss map $G:M\to S^2$ of the minimal surface and 
the stereographic projection 
$\sigma:S^2\to\mathbb{C}\cup\{\infty\}$, that is, 
$g=\sigma\circ G$.
So we call $g$ the Gauss map of the minimal surface. 
\end{remark}

\begin{remark}[A historical remark about triply periodic minimal surfaces]\label{rm:tpms}
The first examples of triply periodic minimal surfaces in $\mathbb{R}^3$ are 
found by H.~A.~Schwarz in the 19th century \cite{Schw}. 
Then in 1970, a NASA scientist A.~Schoen found many more examples, 
and he named three of Schwarz' examples P surface, D surface, and H family, because 
they have the symmetry related to those of the {\it primitive} cubic lattice, 
{\it diamond} crystal structure, and {\it hexagonal} crystal structure, 
respectively \cite{Sch}. 

In 1989, H.~Karcher found a 1-parameter family of triply periodic 
minimal surfaces \cite{Ka1}. 
Since a half of the fundamental piece of the surface looks like
{\it twisted} (equilateral) {\it trianglar} catenoid (see Fig.~\ref{fg:rPD}), 
he named the family TT, but since the family contains both Schwarz P and D surfaces, 
the family is now called rPD family.
See for example \cite{FK}.
\end{remark}

\begin{example}[Schwarz rPD family]
\label{sbsec:rpd}
For a constant $a\in (0,\infty)$, we set $M_a$ a Riemann surface 
of genus 3 defined by the hyperelliptic curve
\[
  w^2=z(z^3-a^3)(z^3+a^{-3}).
\]
We define the Weierstrass data
\begin{equation}\label{eq:w-data}
  g=z,\qquad \eta=\frac{dz}{w}.
\end{equation}
Then 
\begin{equation}\label{eq:w-rep-min}
 f_a=\begin{pmatrix}x_1 \\ x_2 \\ x_3\end{pmatrix}
    =\Re\int
   \begin{pmatrix}
     1-g^2 \\ i(1+g^2) \\ 2g
   \end{pmatrix} \eta
\end{equation}
gives a 1-parameter family $\{f_a\}_{0<a<\infty}$ of embedded triply periodic 
minimal surfaces in $\mathbb{R}^3$. 
This family is called the {\it Schwarz rPD family}. 
When $a=1/\sqrt{2}$, the surface coincides with Schwarz P surface, 
and when $a=\sqrt{2}$, the surface coincides with Schwarz D surface. 

As we mentioned in Remark~\ref{rm:tpms}, 
a half of the fundamental piece of rPD surface looks like 
``twisted'' equilateral triangular catenoid.  
See Fig.~\ref{fg:rPD}. 
\end{example}

\begin{figure}[htbp] 
\begin{center}
\begin{tabular}{cc}
 \includegraphics[width=.35\linewidth]{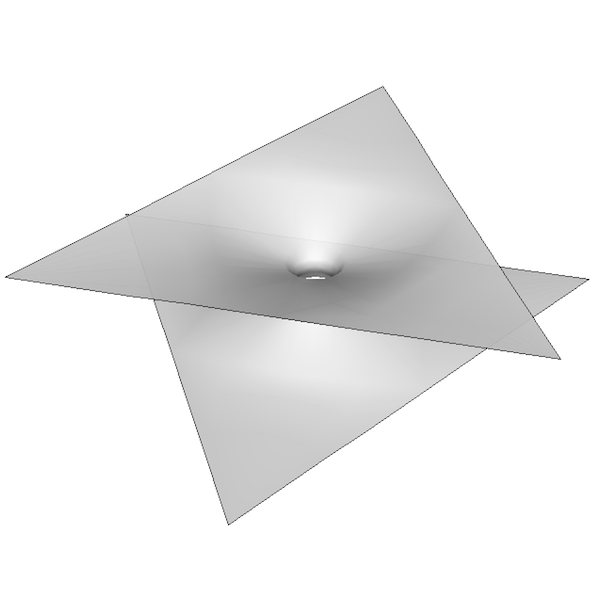} &
 \includegraphics[width=.35\linewidth]{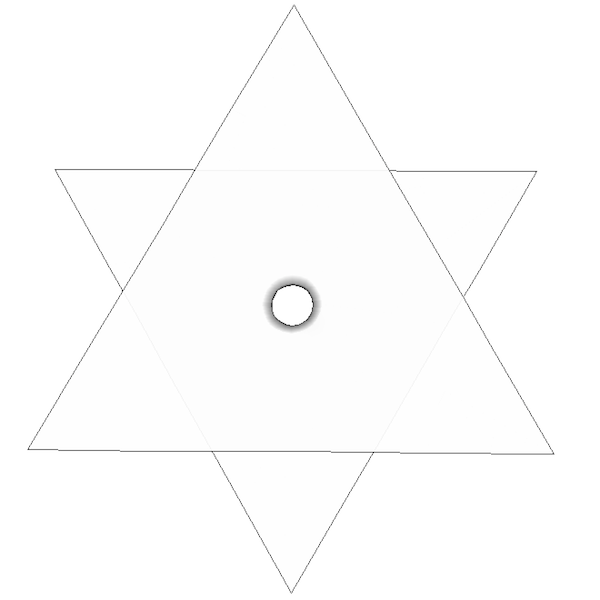} \\
 \multicolumn{2}{c}{$a=0.1$}  \\
 \includegraphics[width=.35\linewidth]{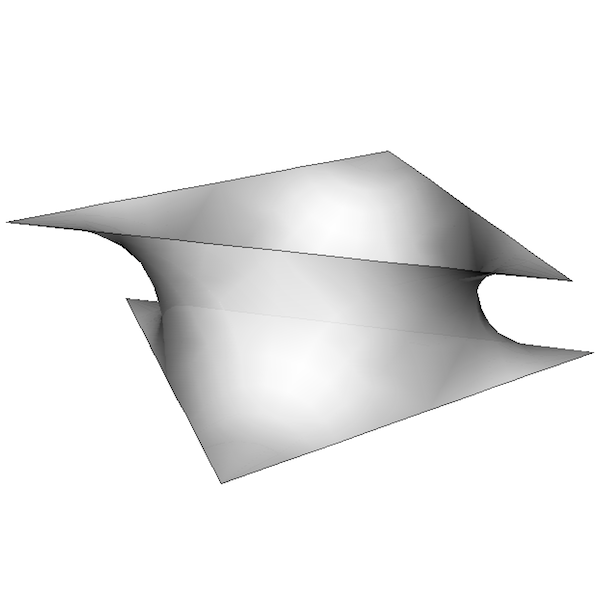} &
 \includegraphics[width=.35\linewidth]{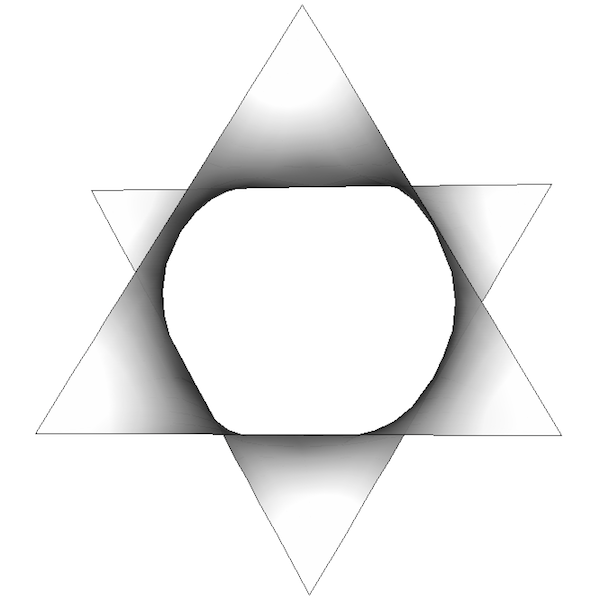} \\
 \multicolumn{2}{c}{$a=1.0$}  \\
 \includegraphics[width=.35\linewidth]{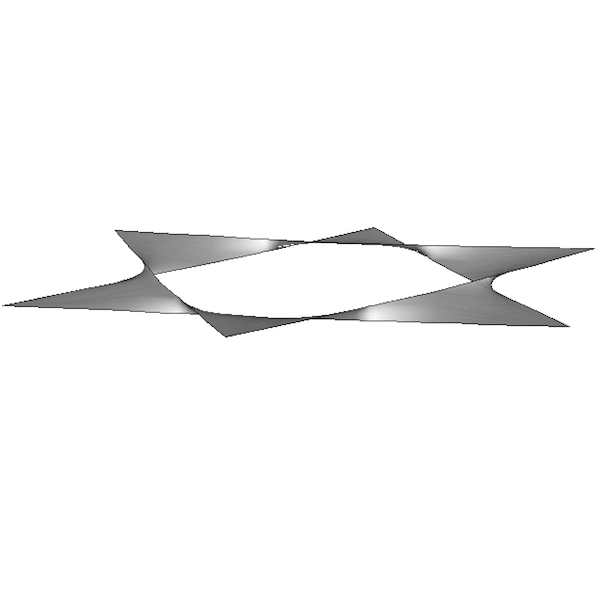} &
 \includegraphics[width=.35\linewidth]{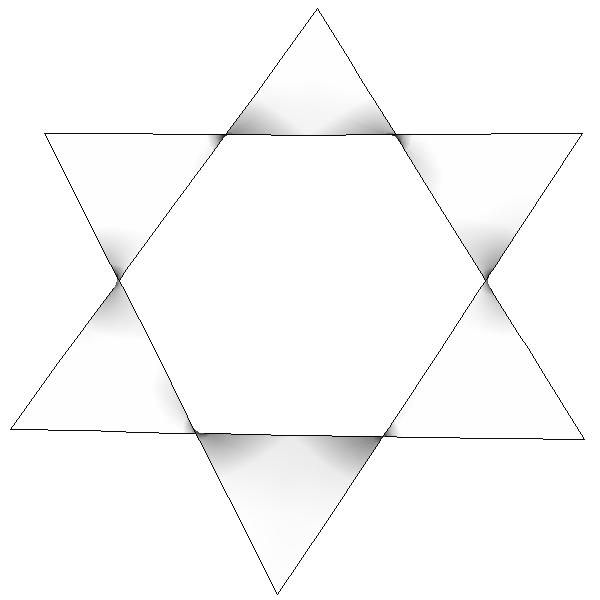} \\
 \multicolumn{2}{c}{$a=10.0$}  
\end{tabular}
\caption{Schwarz rPD surfaces.}
\label{fg:rPD}
\end{center}
\end{figure} 

Fig.~\ref{fg:rpd-p} shows the relation between Schwarz P and rPD for $a=1/\sqrt{2}$. 

\begin{figure}[htbp] 
\begin{center}
\begin{tabular}{cc}
 \includegraphics[width=.35\linewidth]{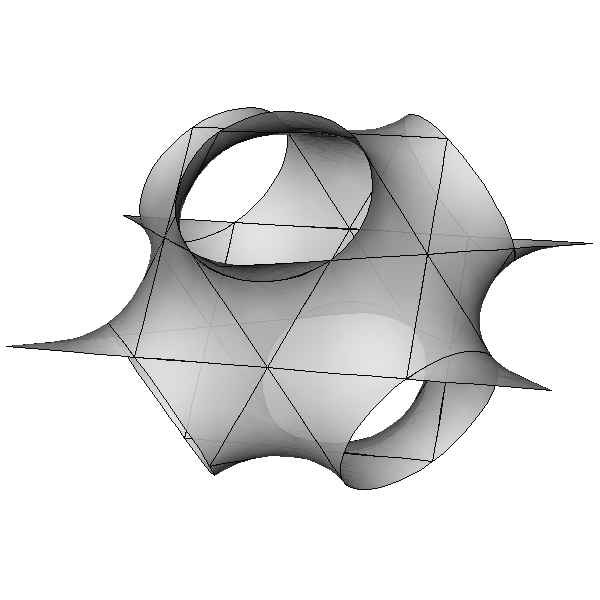} &
 \includegraphics[width=.35\linewidth]{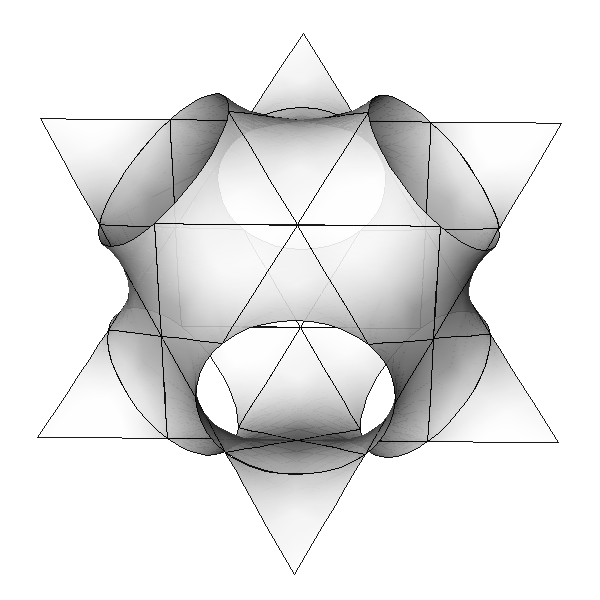} \\
 \multicolumn{2}{c}{$a=1/\sqrt{2}$} 
\end{tabular}
\caption{Relation between Schwarz P and rPD.}
\label{fg:rpd-p}
\end{center}
\end{figure} 

\begin{example}[Schwarz H family]
\label{sbsec:rh}

For a constant $a\in (0,\infty)$, we set $M_a$ a Riemann surface 
of genus 3 defined by the hyperelliptic curve
\[
  w^2=z(z^3+a^3)(z^3+a^{-3}).
\]
Then the family $\{f_a\}_{0<a<1}$ of minimal surfaces \eqref{eq:w-rep-min} 
with the Weierstrass data \eqref{eq:w-data}
is a family of embedded triply periodic minimal surfaces in $\mathbb{R}^3$. 
This family is called the {\it Schwarz H family}. 
As $a\to 0$, $f_a$, with rescaled by $\sqrt{a^3+a^{-3}}$, converges to catenoid. 
Also, as $a\to 1$, $f_a$ converges to Karcher tower with $k=3$ 
(see Example~\ref{sbsec:karcher}).

A half of the fundamental piece of Schwarz H surface looks like 
``non-twisted '' equilateral triangular catenoid.  
See Fig.~\ref{fg:H}. 

Fig.~\ref{fg:Hconj} shows the conjugate surface of Schwarz H surface. 
Each vertex of the hexagon in the right hand side figure lies 
in the straight line parallel to $x_3$-axis.  
Hence after reflections with respect to these lines, 
we see that the surface has self-intersections. 
\end{example}

\begin{figure}[htbp] 
\begin{center}
\begin{tabular}{cc}
 \includegraphics[width=.35\linewidth]{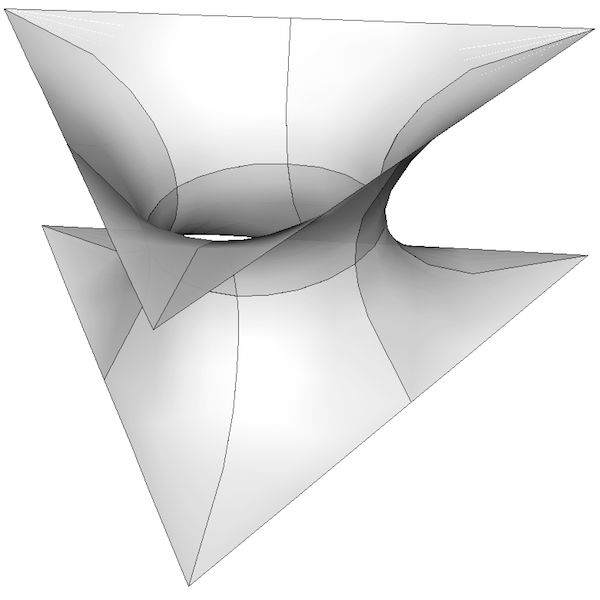} &
 \includegraphics[width=.35\linewidth]{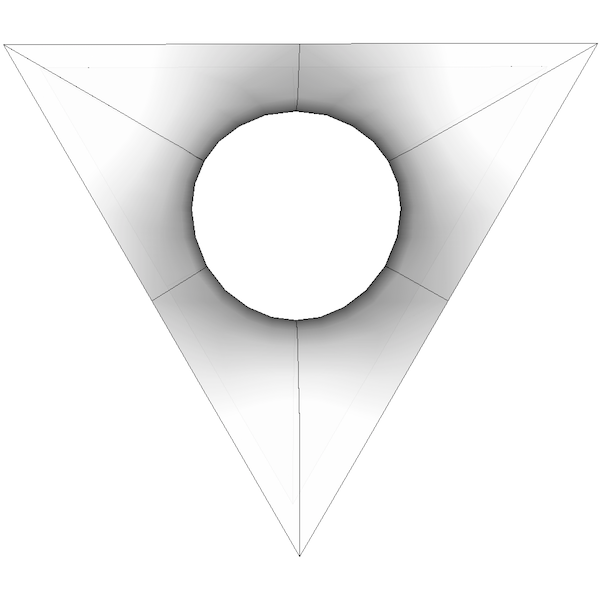} 
\end{tabular}
\caption{Schwarz H surface ($a=0.5$).}
\label{fg:H}
\end{center}
\end{figure} 

\begin{figure}[htbp] 
\begin{center}
\begin{tabular}{cc}
 \includegraphics[width=.35\linewidth]{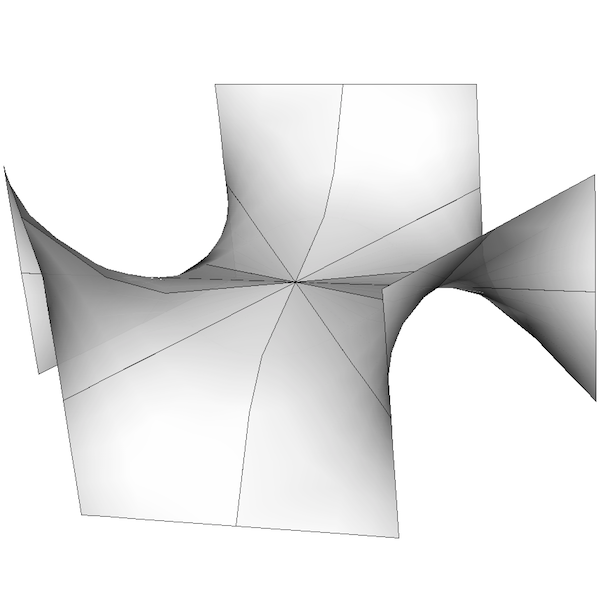} &
 \includegraphics[width=.35\linewidth]{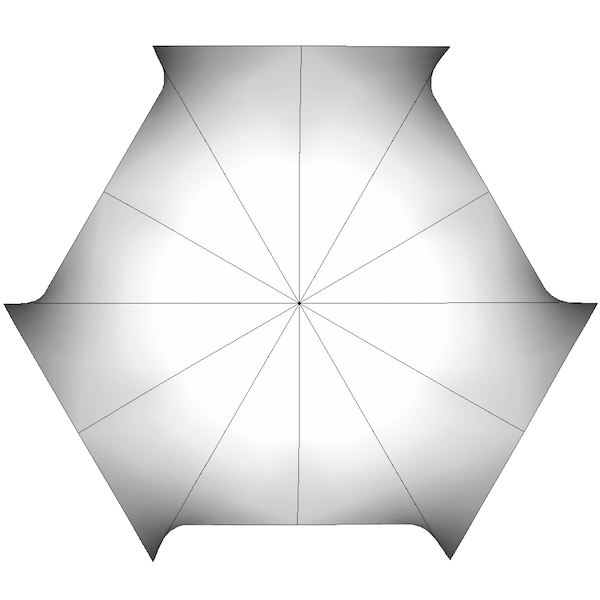} 
\end{tabular}
\caption{The comjugate surface of Schwarz H surface ($a=0.5$).}
\label{fg:Hconj}
\end{center}
\end{figure} 

\begin{example}[Karcher tower]
\label{sbsec:karcher}

For an integer $k\ge 2$, we define $M_k$ by 
\begin{equation}\label{eq:karcher-M}
  M_k=(\mathbb{C}\cup\{\infty\})\setminus
      \{z\in\mathbb{C}\;;\;z^{2k}=-1\}.
\end{equation}
Then the minimal surface \eqref{eq:w-rep-min} with the Weierstrass data 
\[
  g=z^{k-1},\qquad \eta=\frac{dz}{z^{2k}+1}
\]
is embedded singly periodic with $2k$ ends in $\mathbb{R}^3$.  
This minimal surface is called the {\it Karcher tower}.
When $k=2$, this surface coincides with the Scherk tower 
(the conjugate surface of doubly periodic Scherk surface). 
See Fig.~\ref{fg:karcher}. 
\end{example}

\begin{figure}[htbp] 
\begin{center}
\begin{tabular}{cc}
 \includegraphics[width=.30\linewidth]{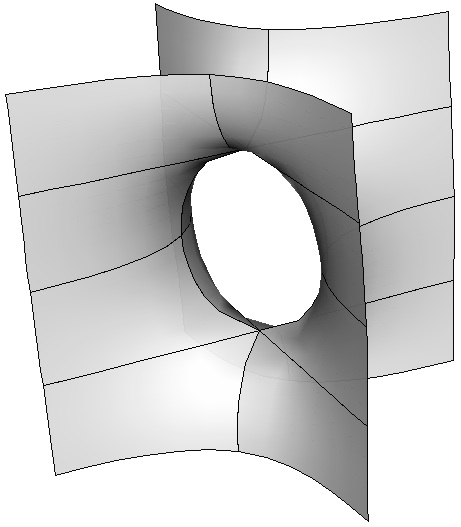} &
 \includegraphics[width=.40\linewidth]{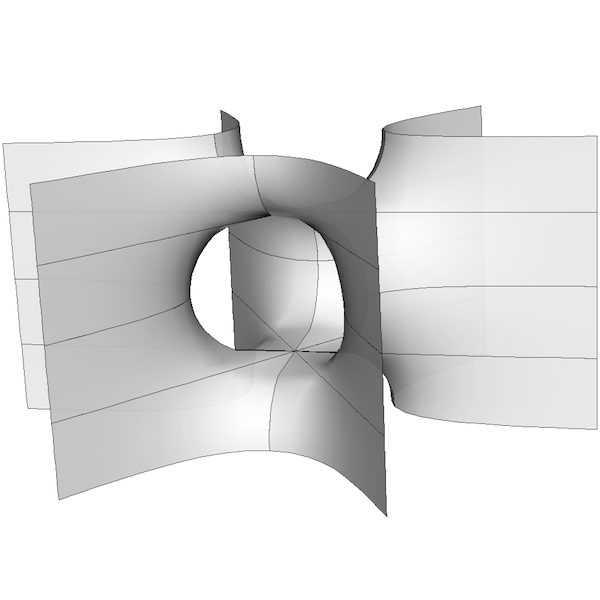} \\
 $k=2$ & $k=3$  
\end{tabular}
\caption{The Karcher tower.}
\label{fg:karcher}
\end{center}
\end{figure} 



\begin{thebibliography}{99}
%
\bibitem{Aka}
S.~Akamine
{\it Causal characters of zero mean curvature surfaces of Riemann-type 
     in the Lorentz-Minkowski 3-space}
preprint, arXiv:1510.07451.
%
\bibitem{CHM}
M.~Callahan, D.~Hoffman, and W.~Meeks III,
{\it The structure of singly-periodic minimal surfaces}, 
Invent. Math. {\bf 99} (1990), 455--481.
%
\bibitem{EFS}
N.~Ejiri, S.~Fujimori, and T.~Shoda, 
{\it A remark on limits of triply periodic minimal surfaces of genus 3}, 
Topology Appl. {\bf 196} (2015), 880--903.
%
\bibitem{FK}
W.~Fischer and E.~Koch, 
{\it Crystallographic aspects of minimal surfaces}, 
Journal de Physique Colloques {\bf 51} (1990), C7-131--C7-147.
%
\bibitem{FKKRSUYY2}
S.~Fujimori, 
Y.~W.~Kim, 
S.-E.~Koh, 
W.~Rossman, 
H.~Shin, 
M.~Umehara,
K.~Yamada, and 
S.-D.~Yang, 
{\it Zero mean curvature surfaces in Lorentz-Minkowski $3$-space 
     and $2$-dimensional fluid mechanics}, 
Math. J. Okayama Univ. {\bf 57} (2015), 173--200.
%
\bibitem{FKKRUY}
S.~Fujimori,
Y.~Kawakami,
M.~Kokubu,
W.~Rossman,
M.~Umehara, and
K.~Yamada, 
{\it  Analytic extension of Jorge-Meeks type maximal surfaces 
      in Lorentz-Minkowski 3-space},
to appear in Osaka J. Math.
%
\bibitem{FKKRUY2}
S.~Fujimori,
Y.~Kawakami,
M.~Kokubu,
W.~Rossman,
M.~Umehara, and
K.~Yamada, 
{\it  Entire zero mean curvature graphs of mixed type 
      in Lorentz-Minkowski 3-space},
Quart. J. Math. {\bf 67} (2016), 801--837.
%
\bibitem{FRUYY2}
S.~Fujimori, 
W.~Rossman, 
M.~Umehara, 
K.~Yamada, and  
S.-D.~Yang, 
{\it New maximal surfaces in Minkowski 3-space with arbitrary genus 
     and their cousins in de Sitter 3-space}, 
Result. Math. {\bf 56} (2009), 41--82.
%
\bibitem{FRUYY3}
S.~Fujimori, 
W.~Rossman, 
M.~Umehara, 
K.~Yamada, and  
S.-D.~Yang, 
{\it Embedded triply periodic zero mean curvature surfaces 
     of mixed type in Lorentz-Minkowski 3-space}, 
Michigan Math. J. {\bf 63} (2014), 189--207. 
%
\bibitem{FSUY}
S.~Fujimori, K.~Saji, M.~Umehara, and K.~Yamada,
 {\itshape Singularities of maximal surfaces},
  Math. Z. {\bf 259} (2008), 827--848.
%
\bibitem{FW}
S.~Fujimori and M.~Weber, 
{\it Triply periodic minimal surfaces bounded by vertical symmetry planes},
Manuscripta Math. {\bf 129} (2009), 29--53.
%
\bibitem{Ka1}
H.~Karcher, 
{\it The Triply Periodic Minimal Surfaces of Alan 
     Schoen and their Constant Mean Curvature Companions}, 
Manuscripta Math. {\bf 64} (1989), 291--357.
%
\bibitem{Ka2}
H.~Karcher, 
{\it Construction of Minimal Surfaces}, 
Surveys in Geometry, 1--96, University of Tokyo, 1989. 
%
\bibitem{K}
 O.~Kobayashi,
 {\it Maximal surfaces in the $3$-dimensional
      Minkowski space $\mathbb{L}^3$},
 Tokyo J. Math. {\bf 6} (1983), 297--309.
%
\bibitem{O}R.~Osserman,
{\it Global properties of minimal surfaces in $E^3$ and $E^n$},
Ann. of Math. (2) {\bf 80} (1964), 340--364.
%
\bibitem{Sch}A.~Schoen, 
 {\it Infinite Periodic Minimal Surfaces without Self-Intersections}, 
 Technical Note D-5541, NASA, Cambridge, Mass. 1970.
%
\bibitem{Schw}H.~A.~Schwarz, 
 {\it Gesammelte Mathematische Abhandlungen}, 
 Springer, 1933.
%
\bibitem{UY}M.~Umehara and K.~Yamada,
{\it Maximal surfaces with singularities in Minkowski space},
Hokkaido Math. J. {\bf 35} (2006), 13--40.
%
\end{thebibliography}
\end{document}